\documentclass[11pt,reqno,oneside,a4paper]{amsart}
\usepackage{bbm}
\usepackage{amsfonts}
\usepackage{mathbbold}

\usepackage{graphicx}
\usepackage{amssymb}
\usepackage{epstopdf}

\usepackage{geometry}
\usepackage{amsmath,amsfonts,amsthm,mathrsfs,amssymb,cite}
\usepackage[usenames]{color}

\newtheorem{thm}{Theorem}[section]

\newtheorem{lem}{Lemma}[section]

\theoremstyle{definition}

\theoremstyle{remark}

\newtheorem{rem}{Remark}[section]
\numberwithin{equation}{section}

\numberwithin{equation}{section}

\newcounter{saveeqn}

\newcommand{\eqnref}[1]{(\ref {#1})}

\newcommand{\bE}{\mathbf{E}}

\newcommand{\bF}{\mathbf{F}}

\newcommand{\bG}{\mathbf{G}}
\newcommand{\bH}{\mathbf{H}}
\newcommand{\Om}{\Omega}
\newcommand{\Bx}{\mathbf{x}}
\newcommand{\By}{\mathbf{y}}

\newcommand{\BU}{\mathbf{U}}

\newcommand{\Acal}{\mathcal{A}}
\newcommand{\Kcal}{\mathcal{K}}

\newcommand{\Scal}{\mathcal{S}}

\newcommand{\Mcal}{\mathcal{M}}

\newcommand{\Ocal}{\mathcal{O}}
\newcommand{\Vcal}{\mathcal{V}}

\newcommand{\ds}{\displaystyle}

\newcommand{\RR}{\mathbb{R}}

\newcommand{\p}{\partial}

\newcommand{\beq}{\begin{equation}}
\newcommand{\eeq}{\end{equation}}

\DeclareMathAlphabet{\itbf}{OML}{cmm}{b}{it}

\def\by{{{\itbf y}}}
\def\bx{{{\itbf x}}}

\title[Recovery of embedded obstacles and surrounding mediums in EM scattering]{Recovery of an embedded obstacle and the surrounding medium for Maxwell's system}

\author{Youjun Deng}
\address{School of Mathematics and Statistics, Central South University, Changsha, Hunan, China.}
\email{youjundeng@csu.edu.cn, dengyijun\_001@163.com}

\author{Hongyu Liu}
\address{Department of Mathematics, Hong Kong Baptist University, Kowloon, Hong Kong SAR, China}

\email{hongyu.liuip@gmail.com}

\author{Xiaodong Liu}
\address{Institute of Applied Mathematics, Academy of Mathematics and Systems Science, Chinese Academy of Sciences, 100190 Beijing, China.}
\email{xdliu@amt.ac.cn}

\date{} 
\begin{document}
\maketitle

\begin{abstract}
In this paper, we are concerned with the inverse electromagnetic scattering problem of recovering a complex scatterer by the corresponding electric far-field data. The complex scatterer consists of an inhomogeneous medium and a possibly embedded perfectly electric conducting (PEC) obstacle. The far-field data are collected corresponding to incident plane waves with a fixed incident direction and a fixed polarisation, but frequencies from an open interval. It is shown that the embedded obstacle can be uniquely recovered by the aforementioned far-field data, independent of the surrounding medium. Furthermore, if the surrounding medium is piecewise homogeneous, then the medium can be recovered as well. Those unique recovery results are new to the literature. Our argument is based on low-frequency expansions of the electromagnetic fields and certain harmonic analysis techniques.

\vspace{.2in}



\end{abstract}

\section{Introduction}

In this paper, we are concerned with the inverse problem of recovering an unknown/in-accessible scatterer by the associated electromagnetic wave probing. The electromagnetic scattering is governed by the time-harmonic Maxwell system. This problem serves as a prototype model to many important inverse problems arising in scientific and technological applications \cite{Amm3,CK,Uhl}. We consider the recovery in a complex scenario where the scatterer consists of an inhomogeneous medium and a possibly embedded impenetrable obstacle. The corresponding study becomes radically more challenging compared to the existing ones, where the recovery is mainly concerned with either one of the obstacle and the medium by assuming the other one is known. In what follows, we first present the mathematical setup of our study by introducing the time-harmonic Maxwell system.

%

Let $\Omega$ and $B$ be bounded domains in $\mathbb{R}^3$ such that $\overline{B}\Subset \Omega$, $\Omega\setminus\overline{B}$ and $\mathbb{R}^3\backslash\overline{\Omega}$ are connected. It is assumed that $\partial B$ is Lipschitz continuous. Physically, $B$ denotes a perfectly electric conducting (PEC) obstacle that is embedded inside an inhomogeneous medium in $\Omega\backslash\overline{B}$.
The electromagnetic (EM) medium is characterised by the electric permittivity $\epsilon$, magnetic permeability $\mu$ and electric conductivity $\sigma$. Throughout, we assume that $\mu=\mu_0$ with a positive constant $\mu_0\in\mathbb{R}_+$ and $\sigma=0$.
In the homogeneous background space $\mathbb{R}^3\backslash\overline{\Omega}$, $\epsilon=\epsilon_0$ with a positive constant $\epsilon_0\in\mathbb{R}_+$.
Define $k_0:=\omega\sqrt{\epsilon_0\mu_0}$ to be the wavenumber with respect to a frequency $\omega\in\mathbb{R}_+$.

Introduce the electromagnetic plane waves as follows
$$
\bE^i(\Bx,\omega,\mathbf{d},\mathbf{q}):=\frac{1}{\sqrt{\epsilon_0}}\mathbf{p} e^{ik_0\Bx\cdot \mathbf{d}},\quad \bH^i(\Bx,\omega,\mathbf{d},\mathbf{q}):=\frac{1}{\sqrt{\mu_0}}(\mathbf{d}\times\mathbf{p}) e^{ik_0\Bx\cdot \mathbf{d}},
$$
where $\mathbf{d}\in \mathbb{S}^2:=\{\Bx\in \RR^3: \|\Bx\|=1\}$ signfies the incident direction and $\mathbf{p}\in\mathbb{R}^3$ with $\mathbf{p}\bot\mathbf{d}$ denotes a polarisation vector. Let $\bE(\Bx,\omega,\mathbf{d},\mathbf{q})$ and $\bH(\Bx,\omega,\mathbf{d},\mathbf{q})$, respectively, denote the total electric and magnetic field and they satisfy the time-harmonic Maxwell equations
\begin{equation}\label{eq:pss}
\left \{
 \begin{array}{ll}
\nabla\times{\bE}-i\omega\mu_0{\bH}=0,  &\mbox{in}\,\, \RR^3\setminus \overline{B},\\
\nabla\times{\bH}+i\omega\epsilon {\bE}=0, &\mbox{in}\,\, \RR^3\setminus \overline{B},
 \end{array}
 \right .
\end{equation}
along with the following PEC boundary condition,
\beq\label{eq:radia3}
\nu\times \bE=0, \quad \mbox{on}\,\, \p B,
\eeq
where $\nu$ signifies the exterior unit normal vector to $\partial B$.
The scattered field $\bE^s:=\bE-\bE^i,\, \bH^s:=\bH-\bH^i$ satisfies the Silver-M\"{u}ller radiation condition
\beq\label{eq:radia2}
\lim_{\|\Bx\|\rightarrow\infty} \|\Bx\| \big( \sqrt{\mu_0}\bH^s \times\hat{\Bx}- \sqrt{\epsilon_0}\bE^s\big)=0,
\eeq
where $\hat{\Bx} = \Bx/\|\Bx\|$ for $\mathbf{x}\in\mathbb{R}^3\backslash\{0\}$.
The radiation condition (\ref{eq:radia2}) characterises the outgoing nature of the EM fields and it also implies the following asymptotic expansion of the scattered electric wave,
\begin{equation}\label{eq:ccc1}
\bE^s(\Bx,\omega,\mathbf{d},\mathbf{q})=\frac{e^{ik\|\Bx\|}}{\|\Bx\|}\left\{\bE^\infty(\hat{\Bx},\omega,\mathbf{d},\mathbf{q})+\Ocal\Big(\frac{1}{\|\Bx\|}\Big)\right\}, \qquad\mbox{as}\;\;\|\Bx\|\rightarrow\infty,
\end{equation}
which holds uniformly in all directions $\hat{\Bx}$. The vector field $\bE^{\infty}$ in \eqref{eq:ccc1} defined on the unit sphere $\mathbb{S}^{2}$ is
usually referred to as the electric far-field pattern.

Associated with the Maxwell system described above, the inverse scattering problem that we are concerned with is to recover $B$ and $(\Omega\backslash\overline{B}, \epsilon)$ by knowledge of the far-field pattern $\bE^{\infty}(\hat{\Bx},\omega,\mathbf{d},\mathbf{q})$ for all observation directions $\hat{\Bx}\in\mathbb{S}^2$ and all frequencies in any open interval,
but a fixed incident direction $\mathbf{d}\in\mathbb{S}^2$ and a fixed polarisation $\mathbf{q}\in\mathbb{R}^3$.
It is noted that $\bE^{\infty}(\hat{\Bx},\omega,\mathbf{d},\mathbf{q})$ is (real) analytic in all of its arguments (cf. \cite{CK}), and hence if the far-field pattern is known for $\hat{\Bx}$ from an open subset of $\mathbb{S}^{2}$, then it is known on the whole sphere $\mathbb{S}^{2}$. The same remark holds equally for the frequency $\omega$, the incident direction $\mathbf{d}$ and the polarization $\mathbf{q}$.

There is a fertile mathematical theory for the inverse scattering problem described above.
In this work, we shall be mainly concerned with the unique recovery or identifiability issue; that is, given the measurement data, what kind of unknowns that one can recover.
The unique recovery of solely a PEC obstacle $B$, namely without the presence of the surrounding inhomogeneous medium, by knowledge of $\bE^\infty(\hat{\Bx}, \omega, \mathbf{d},\mathbf{q})$ for either i)~all $\hat{\Bx}$, $\mathbf{d}$ and $\mathbf{q}$ along with a fixed $\omega$; or ii)~all $\hat{\Bx}$ and $\omega$
along with a fixed $\mathbf{d}$ and $\mathbf{q}$ can be found in\cite{CK}. If the obstacle $B$ is of general polyhedral type, the uniqueness results of recovering $B$ can be found in \cite{Liu1,LRX,LYZ}. Without the presence of the embedded obstacle, the uniqueness in recovering solely an inhomogeneous medium by $\bE^\infty(\hat{\Bx}, \omega, \mathbf{d})$ for all $\hat{\Bx}$, $\mathbf{d}$ and $\mathbf{q}$ along with any fixed $\omega$ has been established in \cite{ColtonPaivarinta,18}. The recovery of a complex scatterer as described earlier consisting of both an embedded obstacle and a surrounding inhomogeneous medium was considered in \cite{LiuZhang09AA}. But the study therein is to recover the obstacle by assuming that the surrounding medium is piecewise homogeneous and known a priori. To our best knowledge, there is no unique recovery result available in the literature in simultaneously recovering both $B$ and $(\Omega\backslash\overline{B}, \epsilon)$. It is interesting to note that the simultaneous recovery is also closely related to the partial data inverse boundary value problem in electrodynamics \cite{COS}. Furthermore, in the current article, we consider the recovery by knowledge of $\bE^\infty(\hat{\Bx}, \omega, \mathbf{d}, \mathbf{q})$ for all $\hat{\Bx}$ and $\omega$, but any fixed $\mathbf{d}$ and  $\mathbf{q}$. Compared to the frequently used far-field data in the literature with $\bE^\infty(\hat{\Bx}, \omega, \mathbf{d}, \mathbf{q})$ for all $\hat{\Bx}$, $\mathbf{d}$ and $\mathbf{q}$, but a fixed $\omega$, the scattering information used in our study is obviously diminishing. We establish that that the embedded obstacle can be uniquely recovered by the aforementioned far-field data, independent of the surrounding medium. Furthermore, if the surrounding medium is piecewise homogeneous, then the medium can be recovered as well.

Finally, we briefly discuss the mathematical arguments to establish our unique recovery results. Our idea follows from a recent work \cite{LiuLiu17} by two of the authors where the acoustic case was considered.
First, we derive the integral representation of the solution to the scattering problem involving both the perfect conductor $B$ and the medium $(\Omega\backslash\overline{B}, \epsilon)$.
Then, by considering the low wavenumber asymptotics in terms of $\omega$, we can derive some integral identities,
which can serve to decouple the scattering information of $B$ from that of $(\Omega\backslash\overline{B}, \epsilon)$.
Finally, by using certain harmonic analysis techniques, we can invert the previously obtained integral identities to recover the conductor and the medium.
Our study heavily relies on the low-frequency asymptotics of the underlying scattering problem.
We refer to Dassios and Kleinman   \cite{DK}  and the references therein for relevant results on low-frequency asymptotics for various scattering problems;
and we also refer to Ammari and Kang \cite{AK04} for results on asymptotics of scattering from small inhomogeneities, which are also related to our current study.
However, we would like to emphasise that our results on the low-frequency asymptotics of scattering from an inhomogeneous medium containing an obstacle are new to the literature. Compared to the acoustic study in \cite{LiuLiu17}, the arguments for the electromagnetic case are much more subtle and technical, and we believe the arguments developed in this work can be used to deal with other EM scattering problems.

The rest of the paper is organised as follows. In Section 2, we present some preliminary knowledge on the boundary layer potentials and volume potentials.
Section 3 is devoted to the study of the forward scattering problem. In Section 4, we present the simultaneous recovery results.

\section{Preliminaries on potential operators}

For our study of the inverse problem, we shall solve the boundary value problem \eqref{eq:pss}-\eqnref{eq:radia2} by the integral equation method. To that end, we recall some preliminary knowledge on the function spaces and the potential operators used in the context of Maxwell's equations. We also refer the reader to \cite{CK,Mclean,Ned} for more relevant details. In what follows, for any bounded domain $\mathcal{U}\subset \RR^3$ with a Lipschitz boundary $\p \mathcal{U}$, we denote by $\nu$ the unit outward normal to $\p \mathcal{U}$.

\subsection{Function spaces}
As usual, $L^2(\mathcal{U})$ denotes the set of all square integrable functions on $\mathcal{U}$.
Let $X$ be a domain in $\RR^3$. In what follows, we introduce the following vector spaces
\begin{eqnarray*}
H(\mbox{curl}; \mathcal{U})&:=&\{ \BU\in (L^2(\mathcal{U}))^3;\ \nabla\times \BU\in (L^2(\mathcal{U}))^3 \},\\
H_{loc}(\mbox{curl}; X)&:=&\{ \BU|_\mathcal{U}\in H(\mbox{curl}; \mathcal{U});\ \mathcal{U}\ \ \mbox{is any bounded subdomain of $X$} \},\\
H(\mbox{div}; \mathcal{U})&:=&\{ \BU\in (L^2(\mathcal{U}))^3;\ \nabla\cdot \BU\in L^2(\mathcal{U}) \},\\
H_{loc}(\mbox{div}; X)&:=&\{ \BU|_\mathcal{U}\in H(\mbox{div}; \mathcal{U});\ \mathcal{U}\ \ \mbox{is any bounded subdomain of $X$} \},\\
H(\mbox{div}(\beta\cdot); \mathcal{U})&:=&\{ \BU\in (L^2(\mathcal{U}))^3; \quad \nabla\cdot (\beta \BU)\in L^2(\mathcal{U})\},
\end{eqnarray*}
where $\beta\in L^\infty(\mathcal{U})$.

We need also some spaces for vector fields on the boundary.
Let us first define the space of tangential vector fields by
\begin{eqnarray*}
L_T^2(\p \mathcal{U}):=\{\Phi\in (L^2(\p \mathcal{U}))^3;\, \nu\cdot \Phi=0\}.
\end{eqnarray*}
Denote by $\nabla_{\p \mathcal{U}}\cdot$ the standard surface divergence operator defined on $L_T^2(\p \mathcal{U})$.
We then introduce the normed spaces of tangential fields by
\begin{align*}
\mathrm{TH}({\rm div}, \p \mathcal{U}):&=\Bigr\{ {\Phi} \in L_T^2(\partial \mathcal{U});\,
\nabla_{\partial \mathcal{U}}\cdot {\Phi} \in L^2(\partial \mathcal{U}) \Bigr\},\\
\mathrm{TH}({\rm curl}, \p \mathcal{U}):&=\Bigr\{ {\Phi} \in L_T^2(\partial \mathcal{U});\,
\nabla_{\partial \mathcal{U}}\cdot ({\Phi}\times {\nu}) \in L^2(\partial \mathcal{U}) \Bigr\},
\end{align*}
equipped with the norms
\begin{align*}
&\|{\Phi}\|_{\mathrm{TH}({\rm div}, \p \mathcal{U})}=\|{\Phi}\|_{(L^2(\p \mathcal{U}))^3}+\|\nabla_{\p \mathcal{U}}\cdot {\Phi}\|_{L^2(\p \mathcal{U})}, \\
&\|{\Phi}\|_{\mathrm{TH}({\rm curl}, \p \mathcal{U})}=\|{\Phi}\|_{(L^2(\p \mathcal{U}))^3}+\|\nabla_{\p \mathcal{U}}\cdot({\Phi}\times \nu)\|_{L^2(\p \mathcal{U})}.
\end{align*}
Finally, let $H^s(\partial \mathcal{U})$ be the usual Sobolev space of order $s$ on $\partial \mathcal{U}$.

\subsection{Volume and surface potentials}


Denote by $\Gamma_{k_0}$ the fundamental solution of the Helmholtz equation with wave number $k_{0}$, which is given by
\begin{equation}\label{Gk} \ds
\Gamma_{k_0} (\Bx) = \frac{e^{ik_0 \|\Bx\|}}{4 \pi \|\Bx\|}, \quad \Bx\neq \mathbbold{0}.
 \end{equation}
The volume potential operator $\Vcal_{\mathcal{U}}^{k_0}: (L^2(\mathcal{U}))^3\rightarrow (H^2(\mathcal{U}))^3$ is defined by
\begin{equation*}
\Vcal_{\mathcal{U}}^{k_0}[\Phi](\Bx):=\int_{\mathcal{U}}\Gamma_{k_0}(\Bx-\By)\Phi(\By)d\By,\quad \Bx\in \mathcal{U}.
\end{equation*}
We also denote by $\Scal_\mathcal{U}^{k_0}: H^{-1/2}(\p \mathcal{U})\rightarrow H_{loc}^{1}(\RR^3)$ the single layer potential given by
\begin{equation*}
\Scal_{\mathcal{U}}^{k_0}[\phi](\Bx):=\int_{\p \mathcal{U}}\Gamma_{k_0}(\Bx-\By)\phi(\By)d s_\By, \quad \Bx\in \RR^3,
\end{equation*}
and $\Kcal_\mathcal{U}^{k_0}: H^{1/2}(\p \mathcal{U})\rightarrow H^{1/2}(\p \mathcal{U})$ the Neumann-Poincar\'e operator
\begin{equation*}
\Kcal_{\mathcal{U}}^{k_0}[\phi](\Bx):=\mbox{p.v.}\quad\int_{\p \mathcal{U}}\frac{\p\Gamma_{k_0}(\Bx-\By)}{\p \nu_y}\phi(\By)d s_\By,\quad \Bx\in\p \mathcal{U},
\end{equation*}
where p.v. stands for the Cauchy principle value.
It is known that the single layer potential $\Scal_\mathcal{U}^{k_0}$ satisfies the trace formula
\begin{equation*}
\frac{\p}{\p\nu}\Scal_\mathcal{U}^{k_0}[\phi] \Big|_{\pm}
= \Big(\mp \frac{1}{2}I+(\Kcal_{\mathcal{U}}^{k_0})^*\Big)[\phi] \quad \mbox{on } \p \mathcal{U},
\end{equation*}
where $(\Kcal_{\mathcal{U}}^{k_0})^*$ is the adjoint operator of $\Kcal_\mathcal{U}^{k_0}$.
In addition, for a density $\Phi \in \mathrm{TH}(\mbox{div}, \p \mathcal{U})$, we also define the
vectorial single layer potential by
\begin{equation*}
\mathcal{A}_\mathcal{U}^{k_0}[\Phi](\Bx) := \int_{\p \mathcal{U}} \Gamma_{k_0}(\Bx-\By)
\Phi(\By) d s_\By, \quad \Bx \in \RR^3.
\end{equation*}
It is known that $\nabla\times\mathcal{A}_\mathcal{U}^{k_0}$ satisfies the following jump formula
\begin{equation}\label{jumpM}
\nu \times \nabla \times \mathcal{A}_\mathcal{U}^{k_0}[\Phi]\big\vert_\pm = \pm \frac{\Phi}{2} + \mathcal{M}_\mathcal{U}^{k_0}[\Phi] \quad \mbox{ on } \p \mathcal{U},
\end{equation}
where
\begin{equation*}
\quad \nu \times \nabla \times \mathcal{A}_\mathcal{U}^{k_0}[\Phi]\big\vert_\pm (\Bx)= \lim_{t\rightarrow 0^+} \nu \times \nabla \times \mathcal{A}_\mathcal{U}^{k_0}[\Phi] (\Bx\pm t \nu), \quad \Bx\in \p \mathcal{U}
\end{equation*}
is understood in the sense of uniform convergence on $\p \mathcal{U}$ and the boundary integral operator
$\mathcal{M}^{k_0}_\mathcal{U}: \mathrm{TH}(\mbox{div}, \p \mathcal{U})\rightarrow \mathrm{TH}(\mbox{div}, \p \mathcal{U})$ is given by
\begin{equation*}
\mathcal{M}^{k_0}_\mathcal{U}[\Phi](\Bx):= \mbox{p.v.}\quad\nu  \times \nabla \times \int_{\p \mathcal{U}} \Gamma_{k_0}(\Bx,\By) \Phi(\By) d s_\By, \quad \Bx\in \p \mathcal{U}.
\end{equation*}

\section{Integral representation and low frequency expansion}
Based on the surface and volume potentials defined in the previous section, we next use the integral equation method to solve the exterior boundary value problem \eqref{eq:pss}-\eqref{eq:radia2}.

Define a $6\times 6$ matrix function $\bG$ as follows,
\begin{equation*}
\bG(\Bx)= \left(
\begin{array}{ll}
(k_0^2I_{3}+\nabla^2)\Gamma_{k_0}(\Bx) & i\omega\mu_0\nabla\times (\Gamma_{k_0}(\Bx)I_3)\\
-i\omega\epsilon_0\nabla\times(\Gamma_{k_0}(\Bx)I_3) & (k_0^2I_{3}+\nabla^2)\Gamma_{k_0}(\Bx)
\end{array}
\right),
\end{equation*}
where $I_3$ is the $3\times 3$ identity matrix and $\nabla^2$ is the Hessian matrix.
The matrix function $\bG$ is the fundamental solution to the homogeneous Maxwell's equations \cite{18}.
Based on this, we look for a solution to \eqnref{eq:pss}-\eqnref{eq:radia2} in $\RR^3\setminus \overline{B}$ of the form
\beq\label{eq:repre1}
\left(
\begin{array}{l}
\bE\\
\bH
\end{array}
\right)
=\left(
\begin{array}{l}
\bE^i\\
\bH^i
\end{array}
\right)
+\left(
\begin{array}{l}
\bE_0\\
\bH_0
\end{array}
\right)
+\int_{\RR^3\setminus\overline{B}}\bG(\cdot-\By)
\left(
\begin{array}{c}
\tilde{\epsilon}(\By)\bE(\By)\\
0
\end{array}
\right)
d\By,
\eeq
where
\begin{equation}\label{eq:defnew}
\tilde{\epsilon}:=(\epsilon-\epsilon_0)/\epsilon_0
\end{equation}
and the additional field $(\bE_0, \bH_0)$ in \eqnref{eq:repre1} takes the following form
\beq\label{eq:soluformPEC01}
\bE_0=\nabla\times\Acal_B^{k_0}[\Phi_E], \quad \bH_0=-i/(\omega\mu_0)\nabla\times\nabla\times \Acal_B^{k_0}[\Phi_E]
\eeq
for some $\Phi_E\in \mathrm{TH}(\mbox{div}, \p B)$. Clearly, the field $(\bE_0, \bH_0)$ is a radiating solution to the homogeneous Maxwell system in $\RR^3\setminus \overline{B}$.
Note that $\tilde\epsilon$ is compactly supported in $\Omega$. We can rewrite \eqnref{eq:repre1} in the following form
\beq\label{eq:eq:repre1}
\left(
\begin{array}{l}
\bE-\bE_0\\
\bH-\bH_0
\end{array}
\right)-\mathfrak{G}_{\Omega\setminus\overline{B}}^{k_0}
\left(
\begin{array}{l}
\tilde\epsilon\bE\\
0
\end{array}
\right)
=
\left(
\begin{array}{l}
\bE^i\\
\bH^i
\end{array}
\right),
\eeq
where the matrix operator $\mathfrak{G}_{\Omega\setminus\overline{B}}^{k_0}: \big(L^2(\Omega\setminus\overline{B})\big)^3\times \big(L^2(\Omega\setminus\overline{B})\big)^3 \rightarrow \big(L^2(\Omega\setminus\overline{B})\big)^3\times \big(L^2(\Omega\setminus\overline{B})\big)^3$ is defined by
\begin{equation*}
\mathfrak{G}_{\Omega\setminus\overline{B}}^{k_0}:=
\left(
\begin{array}{cc}
(k_0^2I_{3}+D^2)\Vcal_{\Omega\setminus\overline{B}}^{k_0} & i\omega\mu_0\nabla\times\Vcal_{\Omega\setminus\overline{B}}^{k_0}\\
-i\omega\epsilon_0\nabla\times\Vcal_{\Omega\setminus\overline{B}}^{k_0} & (k_0^2I_{3}+D^2)\Vcal_{\Omega\setminus\overline{B}}^{k_0}
\end{array}
\right)
\end{equation*}
with $D^2:=\nabla\nabla\cdot$.
By combining the equations \eqnref{eq:soluformPEC01}-\eqnref{eq:eq:repre1} and the PEC boundary condition \eqnref{eq:radia3}, together with the jump formula \eqnref{jumpM}, we obtain the following equations
\beq\label{eq:reconEq1}
\left \{
 \begin{array}{ll}
(I_3-\mathfrak{D}\Vcal_{\Omega\setminus\overline{B}}^{k_0}\tilde\epsilon)[\bE]
-\nabla\times\Acal_B^{k_0}[\Phi_E]=\bE^i\quad \mbox{in}\,\,\Omega\setminus\overline{B}, \\
-k_0^2\nabla\times\Vcal_{\Omega\setminus\overline{B}}^{k_0}\tilde\epsilon[\bE]
+i\omega\mu_0\bH-\mathfrak{D}\Acal_B^{k_0}[\Phi_E]=i\omega\mu_0\bH^i\quad \mbox{in}\,\,\Omega\setminus\overline{B},\\
\nu\times\mathfrak{D}\Vcal_{\Omega\setminus\overline{B}}^{k_0}\tilde\epsilon[\bE]+(\frac{I}{2}+\Mcal_B^{k_0})[\Phi_E]=-\nu\times\bE^i \quad \mbox{on}\,\,\p B,
 \end{array}
 \right .
\eeq
where $\mathfrak{D}:=k_0^2I_{3}+D^2$. Recall that ${I}/{2}+\Mcal_B^{k_0}$ is invertible on $\mathrm{TH}(\mbox{div}, \p B)$ when $k_0$ is sufficiently small (see \cite{ADM,BL,T}), one then can obtain from the third equation in \eqnref{eq:reconEq1} that
\beq\label{eq:reconEqsol1}
\begin{split}
\Phi_E=&-\Big(\frac{I}{2}+\Mcal_B^{k_0}\Big)^{-1}\nu\times\bE^i-\Big(\frac{I}{2}+\Mcal_B^{k_0}\Big)^{-1}\nu\times\mathfrak{D}\Vcal_{\Omega\setminus\overline{B}}^{k_0}\tilde\epsilon[\bE].
\end{split}
\eeq
By substituting \eqnref{eq:reconEqsol1} into the first two equations in \eqnref{eq:reconEq1} one can obtain that
\beq\label{eq:reconEq2}
\left \{
 \begin{array}{ll}
\mathfrak{F}_1(\omega,\epsilon)[\bE]=-\nabla\times\Acal_B^{k_0}(\frac{I}{2}+\Mcal_B^{k_0})^{-1}[\nu\times\bE^i]+\bE^i,  \\
\mathfrak{F}_2(\omega,\epsilon)[\bE]+i\omega\mu_0\bH
=-\mathfrak{D}\Acal_B^{k_0}(\frac{I}{2}+\Mcal_B^{k_0})^{-1}[\nu\times\bE^i]+i\omega\mu_0\bH^i,
 \end{array}
 \right.
\eeq
where the operators $\mathfrak{F}_1$ and $\mathfrak{F}_2$ are defined respectively by
\begin{equation*}
\begin{split}
\mathfrak{F}_1(\omega,\epsilon)&=I_3-\mathfrak{D}\Vcal_{\Omega\setminus\overline{B}}^{k_0}\tilde\epsilon+\nabla\times\Acal_B^{k_0}
\Big(\frac{I}{2}+\Mcal_B^{k_0}\Big)^{-1}\nu\times\mathfrak{D}\Vcal_{\Omega\setminus\overline{B}}^{k_0}\tilde\epsilon, \\
\mathfrak{F}_2(\omega,\epsilon)&=-k_0^2\nabla\times\Vcal_{\Omega\setminus\overline{B}}^{k_0}\tilde\epsilon
+\mathfrak{D}\Acal_B^{k_0}\Big(\frac{I}{2}+\Mcal_B^{k_0}\Big)^{-1}\nu\times\mathfrak{D}\Vcal_{\Omega\setminus\overline{B}}^{k_0}\tilde\epsilon.
\end{split}
\end{equation*}
For asymptotic analysis, we first define some operators as follows. Let $\mathfrak{K}_j$  and $\mathfrak{L}_j$, $j=1, 2$, be defined on $H(\mbox{curl}; \Omega\setminus\overline{B})$ by
\begin{equation*}
\begin{split}
\mathfrak{K}_1(\epsilon)&=D^2\Vcal_{\Omega\setminus\overline{B}}^{0}\tilde\epsilon-\nabla\times\Acal_B^{0}
\Big(\frac{I}{2}+\Mcal_B^{0}\Big)^{-1}\nu\times D^2\Vcal_{\Omega\setminus\overline{B}}^{0}\tilde\epsilon, \\
\mathfrak{K}_2(\epsilon)&=-i\mu_0^{-1}D^2\Acal_B^{0}\Big(\frac{I}{2}+\Mcal_B^{0}\Big)^{-1}\nu\times D^2\Vcal_{\Omega\setminus\overline{B}}^{0}\tilde\epsilon,
\end{split}
\end{equation*}
together with
\begin{equation*}
\begin{split}
\mathfrak{L}_1(\epsilon)&=D^2\Vcal_{\Omega\setminus\overline{B}}^{0}\tilde\epsilon
-D^2\Acal_B^{0}\Big(\frac{I}{2}+\Mcal_B^{0}\Big)^{-1}\nu\times\nabla\times\Vcal_{\Omega\setminus\overline{B}}^{0}\tilde\epsilon, \\
\mathfrak{L}_2(\epsilon)&=
i\epsilon_0\Big[\nabla\times\Vcal_{\Omega\setminus\overline{B}}^{0}\tilde\epsilon
-\nabla\times\Acal_B^{0}
\Big(\frac{I}{2}+\Mcal_B^{0}\Big)^{-1}\nu\times \nabla\times\Vcal_{\Omega\setminus\overline{B}}^{0}\tilde\epsilon\Big].
\end{split}
\end{equation*}
We also introduce the following notations,
\beq\label{eq:defopcon01}
\begin{split}
\mathfrak{C}_1[\bE^i]&=-\nabla\times\Acal_B^{0}\Big(\frac{I}{2}+\Mcal_B^{0}\Big)^{-1}[\nu\times\bE^i]+\bE^i, \\
\mathfrak{C}_2[\bE^i]&=i\mu_0^{-1}D^2\Acal_B^{0}\Big(\frac{I}{2}+\Mcal_B^{0}\Big)^{-1}[\nu\times\bE^i]-i\mu_0^{-1}\nabla\times\bE^i.
\end{split}
\eeq
\begin{lem}\label{le:revisele01}
Let $(\bE,\bH)\in H_{loc}(\mbox{curl}; \RR^3\setminus\overline{B})\times H_{loc}(\mbox{curl}; \RR^3\setminus\overline{B})$ be a solution to \eqnref{eq:pss}-\eqnref{eq:radia2}. Then there holds the following asymptotic behavior in $\Omega\setminus\overline{B}$
\beq\label{eq:reconEq3}
\left \{
 \begin{array}{ll}
(I_3-\mathfrak{K}_1(\epsilon))[\bE]
=\mathfrak{C}_1[\bE^i]+\mathcal{R}_1(\omega,\bE),  \\
\mathfrak{K}_2(\epsilon)[\bE]
+\omega\bH
=\mathfrak{C}_2[\bE^i]+\mathcal{R}_2(\omega,\bE),
 \end{array}
 \right.
\eeq
for sufficient small frequency $\omega$. Here, the remainder terms $\mathcal{R}_1(\omega, \bE),\,\mathcal{R}_2(\omega, \bE)\in H(\rm{curl},\Omega\setminus\overline{B})\cap \mbox{H}(\rm{div},\Omega\setminus\overline{B})$ satisfy
\beq\label{eq:remain01}
\mathcal{R}_1(\omega, \bE), \mathcal{R}_2(\omega, \bE)=\Ocal(\omega^2) \quad\mbox{as}\;\;\omega\rightarrow 0.
\eeq
Furthermore, there holds
\beq\label{eq:remainPE02}
\nu\times\mathcal{R}_1(\omega, \bE)=0 \quad \mbox{on}\,\,\p B.
\eeq
\end{lem}
\begin{proof}
By \eqnref{Gk} and Taylor expansions, for any $\Bx\in \RR^3\setminus \{0\}$, one has
$$
\Gamma_{k_0}(\Bx)=\Gamma_{0}(\Bx)+ \frac{i\sqrt{\epsilon_0\mu_0}}{4\pi}\omega+\Ocal(\omega^2)  \quad\mbox{as}\;\;\omega\rightarrow 0.
$$
Thus
$$
\nabla\Gamma_{k_0}(\Bx)=\nabla\Gamma_{0}(\Bx)+\Ocal(\omega^2) \quad\mbox{as}\;\;\omega\rightarrow 0.
$$
From this, using asymptotic expansions w.r.t. $\omega$ for \eqnref{eq:reconEq2}, the asymptotic estimates \eqnref{eq:reconEq3} and \eqnref{eq:remain01}
now follow.

Next by using the definitions of $\mathfrak{K}_1(\epsilon)$ and $\mathfrak{C}_1[\bE^i]$, jump formula \eqnref{jumpM} one can easily find that
\beq\label{eq:revisele01tmp01}
\nu\times\mathfrak{K}_1(\epsilon)[\bE]=\nu\times\mathfrak{C}_1[\bE^i]=0 \quad \mbox{on}\,\, \p B.
\eeq
Thus the homogeneous boundary condition \eqnref{eq:remainPE02} follows by using the PEC condition \eqnref{eq:radia3}.
The proof is complete.
\end{proof}

The following lemmas are of great importance for the subsequent analysis.
\begin{lem}\label{le:revisele02}
There holds the following identity
\beq\label{eq:lerevise0201}
\mathfrak{K}_1(\epsilon)[\bE]=\nabla u \quad \mbox{in}\,\, \RR^3\setminus\overline{B},
\eeq
where $u\in H^1_{loc}(\RR^3\setminus\overline{B})$ and $u(\Bx)=\Ocal(\|\Bx\|^{-1})$ as $\|\Bx\|\rightarrow \infty$.
\end{lem}
\begin{proof}
We first introduce a tangential vector field $\mathbf{h}$ given by
\beq\label{eq:lerevisetmp01h}
\mathbf{h}:=\Big(\frac{I}{2}+\Mcal_B^{0}\Big)^{-1}\nu\times D^2\Vcal_{\Omega\setminus\overline{B}}^{0}[\tilde\epsilon\bE].
\eeq
Then we have
\beq\label{eq:lerevisetmp01Divh0}
\nabla_{\p B}\cdot \mathbf{h}=0 \quad \mbox{on} \,\, \p B.
\eeq
Indeed, from \eqref{eq:lerevisetmp01h} we find that
$$
\Big(\frac{I}{2}+\Mcal_B^{0}\Big)\mathbf{h} = \nu\times D^2\Vcal_{\Omega\setminus\overline{B}}^{0}[\tilde\epsilon\bE].
$$
Furthermore, by applying $\nabla_{\p B}\cdot$ and using the identities
$$\nabla\times \nabla = \mathbbold{0}\quad\mbox{and}\quad\nabla_{\p B}\cdot\Mcal_B^{0}=-(\Kcal_B^0)^*\nabla_{\p B}\cdot,$$
we see that
\begin{eqnarray*}
({I}/{2}-(\Kcal_B^0)^*)\nabla_{\p B}\cdot \mathbf{h}
&=& \nabla_{\p B}\cdot (I/2+\Mcal_B^{0})\mathbf{h}\cr
&=&\nabla_{\p B}\cdot (\nu\times D^2\Vcal_{\Omega\setminus\overline{B}}^{0}[\tilde\epsilon\bE])\cr
&=&-\nu\cdot \nabla\times \nabla \nabla\cdot\Vcal_{\Omega\setminus\overline{B}}^{0}[\tilde\epsilon\bE]\cr
&=&0.
\end{eqnarray*}
Then the equality \eqref{eq:lerevisetmp01Divh0} follows by noting the fact that
${I}/{2}-(\Kcal_B^0)^*$ is invertible on $H^{1/2}(\p B)$ (see, e.g., \cite{ADM}).

We rewrite $\mathfrak{K}_1(\epsilon)[\bE]$ as follows,
\beq\label{eq:lerevisetmp01}
\mathfrak{K}_1(\epsilon)[\bE]=D^2\Vcal_{\Omega\setminus\overline{B}}^{0}[\tilde\epsilon\bE]-\bF
\eeq
with a vector field $\bF$ given by
\beq\label{eq:lerevisetmp01F}
\bF:=\nabla\times\Acal_B^{0}\mathbf{h} \quad\mbox{in}\,\, \RR^3\setminus\overline{B}.
\eeq
Due to the fact that $\nabla\Gamma_{0}(\Bx)=\Ocal(\|\Bx\|^{-2})$ as $\|\Bx\|\rightarrow \infty$ we have
\beq\label{eq:lerevisetmp01Fdecay}
\bF(\Bx)=\Ocal(\|\Bx\|^{-2})\quad\mbox{ as }\,\, \|\Bx\|\rightarrow \infty.
\eeq
From \eqref{eq:revisele01tmp01} we obtain
\begin{equation*}
\nu\times\bF = \nu\times D^2\Vcal_{\Omega\setminus\overline{B}}^{0}[\tilde\epsilon\bE]\quad \mbox{on}\,\, \p B.
\end{equation*}
Using Stokes' Theorem we obtain
\begin{equation*}
\int_{\p B}\nu\cdot\bF ds = 0.
\end{equation*}
Applying $\nabla\times$ on both sides of \eqnref{eq:lerevisetmp01F}, integrating by parts and using the equality \eqref{eq:lerevisetmp01Divh0}, there holds
\beq\label{eq:lerevisetmp05}
\begin{split}
\nabla\times\bF
=&\nabla\nabla\cdot\Acal_B^0[\mathbf{h}]\\
=&\nabla\int_{\p B}\nabla_\Bx\Gamma_0(\Bx-\By)\cdot\mathbf{h}(\By)ds_\By \\
=&\nabla\int_{\p B}\Gamma_0(\Bx-\By)\nabla_{\p B}\cdot\mathbf{h}(\By)ds_\By \\
=&0\,\, \mbox{in} \,\, \RR^3\setminus \overline{B}.
\end{split}
\eeq
Using the operator $\nabla\cdot\nabla\times = \mathbbold{0}$, we obtain
\beq\label{eq:lerevisetmp01Fdiv}
\nabla\cdot\bF=0\  \ \mbox{in} \, \, \RR^3\setminus \overline{B}.
\eeq

The exterior boundary value problem \eqref{eq:lerevisetmp01Fdecay}-\eqref{eq:lerevisetmp01Fdiv} has a unique solution.
In fact, letting $\bF_0$ be the solution to
\beq\label{eq:revisenewsys0001}
\left \{
 \begin{array}{ll}
\displaystyle{\nabla\times{\bF_0}=0, \quad \nabla\cdot\bF_0=0},  &\mbox{in} \quad \RR^3\setminus \overline{B},\\
\displaystyle{\nu\times\bF_0|_{\p B}^+=0}, \quad& \displaystyle{\int_{\p B}\nu\cdot \bF_0 ds=0},\medskip\\
\bF_0=\Ocal(\|\Bx\|^{-2}), & \|\Bx\|\rightarrow \infty,
 \end{array}
 \right .
\eeq
then one has $\Delta\bF_0=0$ in $\RR^3\setminus\overline{B}$ and thus $\bF_0$ has
the following Helmholtz decomposition (see, e.g., \cite{Gri99}) fomula
\begin{equation*}
\bF_0=\nabla v +\nabla\times \Phi \quad \mbox{in} \,\, \RR^3\setminus \overline{B},
\end{equation*}
where
\begin{equation*}
v:=-\Vcal_{\RR^3\setminus\overline{B}}^0[\nabla\cdot\bF_0]+\Scal_B^0 [\nu\cdot\bF_0],
\end{equation*}
and
\beq\label{eq:lerevisetmp0004}
\Phi:=\Vcal_{\RR^3\setminus\overline{B}}^0[\nabla\times\bF_0]-\Acal_B^0 [\nu\times\bF_0].
\eeq
By using \eqnref{eq:revisenewsys0001}-\eqnref{eq:lerevisetmp0004} one has
\beq\label{eq:lerevisetmp0005}
\bF_0=\nabla\Scal_B^0 [\nu\cdot\bF_0].
\eeq
The boundary condition $\nu\times\bF_0=0$ on $\p B$ then requires that $\Scal_B^0[\nu\cdot\bF_0]=C_1$ on $\p B$ for some constant $C_1$. Finally, by using
$\int_{\p B} \nu\cdot \bF_0 ds=0$, one has
\begin{equation*}
\begin{split}
\int_{\RR^3\setminus\overline{B}}|\nabla\Scal_B^0 [\nu\cdot\bF_0]|^2 dx
=&-\int_{\p B} \nu\cdot\nabla\Scal_B^0 [\nu\cdot\bF_0]|_+ \Scal_B^0 [\nu\cdot\bF_0]ds\\
=&C_1\int_{\p B}\nu\cdot\bF_0 ds=0.
\end{split}
\end{equation*}
Thus $\nabla\Scal_B^0 [\nu\cdot\bF_0]=0$ and by \eqnref{eq:lerevisetmp0005} one has $\bF_0=0$, which shows the uniqueness of solution to the exterior boundary value problem
\eqref{eq:lerevisetmp01Fdecay}-\eqref{eq:lerevisetmp01Fdiv}.

On the other hand, let $\tilde{u}$ be the unique solution to the following system
\begin{equation*}
\left \{
 \begin{array}{ll}
\Delta\tilde u=0  &\mbox{in}\,\, \RR^3\setminus \overline{B},\\
 \tilde{u}|_+=\nabla\cdot\Vcal_{\Omega\setminus\overline{B}}^{0}[\tilde\epsilon\bE]|_+-C_2 & \mbox{on} \,\, \p B,\\
\tilde{u}=\Ocal(\|\Bx\|^{-1}) &\mbox{as} \,\,  \|\Bx\|\rightarrow \infty,
 \end{array}
 \right .
\end{equation*}
where $C_2:=\int_{\p B}(\Scal_B^0)^{-1}[\nabla\cdot\Vcal_{\Omega\setminus\overline{B}}^{0}[\tilde\epsilon\bE]]\,ds /\int_{\p B}(\Scal_B^0)^{-1}[1]\,ds$.
One can easily see that $\nabla\tilde u$ also satisfies \eqref{eq:lerevisetmp01Fdecay}-\eqref{eq:lerevisetmp01Fdiv}. Thus $\bF=\nabla\tilde u$.

The proof is complete.
\end{proof}
\begin{rem}
We mention that one can also find similar results for the second term of RHS of \eqnref{eq:lerevisetmp01} in Lemma 5.5 in \cite{ADM} and \cite{griesmaier2008asymptotic}. Additionally, if $B$ is a simply connected domain, or if $\Omega\setminus\overline{B}$ is a simply connected domain, then one can directly have that $\mathfrak{K}_1(\epsilon)[\bE]$ is a gradient field (see, e.g., Theorem 3.37 in \cite{Mon03})). However, in Lemma~\ref{le:revisele02}, we do not require the simple connectedness and moreover, we derive the asymptotic decaying condition of the gradient field at infinity, which shall also be need in our subsequent analysis.
\end{rem}
\begin{lem}\label{le:00}
There holds the following identity
\beq\label{eq:le0001}
\mathfrak{K}_2(\epsilon)=0.
\eeq
\end{lem}
\begin{proof}
\eqnref{eq:le0001} is a straightforward result from \eqnref{eq:lerevisetmp05}.
\end{proof}

\begin{lem}
Let $\Phi\in \rm{TH}(\rm{curl}, \p B)$, then there holds the following identity
\beq\label{eq:le0301}
D^2\Acal_B^{0}({I}/{2}+\Mcal_B^{0})^{-1}[\nu\times \Phi]=0 \quad \mbox{in}\,\, \RR^3,
\eeq
if and only if $\nabla_{\p B}\cdot (\nu\times\Phi)=0$ holds on $\p B$.
\end{lem}
\begin{proof}
Define
\begin{eqnarray*}
\mathbf{h}:=({I}/{2}+\Mcal_B^{0})^{-1}[\nu\times \Phi],
\end{eqnarray*}
Then
$$
\Big({I}/{2}+\Mcal_B^{0}\Big)\mathbf{h} = \nu\times \Phi.
$$
Furthermore, by applying $\nabla_{\p B}\cdot$ and using the identity $\nabla_{\p B}\cdot\Mcal_B^{0}=-(\Kcal_B^0)^*\nabla_{\p B}\cdot$, we obtain that
\begin{eqnarray*}
\Big({I}/{2}-(\Kcal_B^0)^*\Big)\nabla_{\p B}\cdot \mathbf{h}
= \nabla_{\p B}\cdot (I/2+\Mcal_B^{0})\mathbf{h}
= \nabla_{\p B}\cdot (\nu\times \Phi).
\end{eqnarray*}
Note that ${I}/{2}-(\Kcal_B^0)^*$ is invertible on $H^{1/2}(\p B)$, we conclude that
\beq\label{eq:revisetmp0011}
\nabla_{\p B}\cdot (\nu\times \Phi)=0\quad\mbox{if and only if}\quad \nabla_{\p B}\cdot \mathbf{h}=0.
\eeq

From the relation $\nabla_{\bx}\Gamma_{0}(\bx,\by)=-\nabla_{\by}\Gamma_{0}(\bx,\by)$, by Gauss' surface divergence theorem we have
\begin{eqnarray}\label{eq:revisetmp0012}
D^2\Acal_B^{0}({I}/{2}+\Mcal_B^{0})^{-1}[\nu\times \Phi]
&=& \nabla\nabla\cdot\int_{\p B}\Gamma_0(x,y)\mathbf{h}(\by)ds_{\by}\cr
&=& \nabla\int_{\p B}\nabla_{\bx}\Gamma_0(x,y)\cdot\mathbf{h}(\by)ds_{\by}\cr
&=& -\nabla\int_{\p B}\nabla_{\by}\Gamma_0(x,y)\cdot\mathbf{h}(\by)ds_{\by}\cr
&=& \nabla\int_{\p B}\Gamma_0(x,y)\nabla_{\by}\cdot\mathbf{h}(\by)ds_{\by}\cr
&=& \nabla\Scal_B^0[\nabla_{\p B}\cdot\mathbf{h}].
\end{eqnarray}

If  $\nabla_{\p B}\cdot (\nu\times\Phi)=0$, then the equality \eqnref{eq:le0301} follows directly from \eqref{eq:revisetmp0011}-\eqref{eq:revisetmp0012}.

One the other hand, if \eqnref{eq:le0301} holds then the equality \eqref{eq:revisetmp0012} implies
$$
\nabla\Scal_B^0[\nabla_{\p B}\cdot\mathbf{h}]=0\quad \mbox{in}\,\, \RR^3,
$$
and thus $\Scal_B^0[\nabla_{\p B}\cdot\mathbf{h}]=C$ in $\RR^3$ for some constant $C\in \mathbb{C}$. The decay property of single layer potential shows that $C=0$. The invertibility of single layer potential then shows that $\nabla_{\p B}\cdot\mathbf{h}=0$. Then $\nabla_{\p B}\cdot (\nu\times\Phi)=0$ follows by using the result given in \eqref{eq:revisetmp0011}.
The proof is complete.
\end{proof}

By using \eqnref{eq:defopcon01} and \eqnref{eq:le0301} one can easily obtain that
\beq\label{eq:revise001}
\mathfrak{C}_2[\mathbf{p}]=0.
\eeq

Before deriving the full asymptotic expansions of electric field and magnetic filed with respect to the frequency $\omega$, we need to consider the solvability of the operator equations related to $I_3-\mathfrak{K}_1(\epsilon)$ in an appropriate space. We have the following result
\begin{lem}\label{le:01}
Let $\tilde{\epsilon}$  be defined in \eqnref{eq:defnew}. Suppose further that
\beq\label{eq:le01revise001}
\epsilon= \epsilon_1 \quad \mbox{on}\,\, \p B,
\eeq
for some constant $\epsilon_1$.
 For any $\bF\in H_{loc}(\rm{curl}, \RR^3\setminus\overline{B}))\cap \mbox{H}_{loc}(\rm{div(\tilde\epsilon\cdot)}, \RR^3\setminus\overline{B}))$ and $\nu\times\bF=0$ on $\p B$, the operator equation
\beq\label{eq:le01revise01}
(I_3-\mathfrak{K}_1(\epsilon))[\bE]=\bF
\eeq
is uniquely solvable in $H_{loc}(\mbox{curl}; \RR^3\setminus\overline{B})$.
\end{lem}
\begin{proof}
First, by using the fact that $\nu\times\bF=0$ and $\nu\times\mathfrak{K}_1(\epsilon)[\bE]=0$ on $\p B$, one can easily find that the solution $\bE$, if exists, should satisfy $\nu\times\bE=0$ on $\p B$. Furthermore, by \eqnref{eq:le01revise001} and integration by parts, there holds
\begin{equation*}
\int_{\p B}\nu\cdot\epsilon (\bE-\bF)\,ds=\epsilon_1\int_{\p B} \nu\cdot\mathfrak{K}_1(\epsilon)[\bE]\,ds=\epsilon_1\int_{B} \nabla\cdot\mathfrak{K}_1(\epsilon)[\bE]\,dx=0.
\end{equation*}
One thus needs to prove the unique solvability of the following equation
\beq\label{eq:le01tmp01}
\left\{
\begin{split}
(I_3-\mathfrak{K}_1(\epsilon))[\bE]=\bF,\quad & \mbox{in}\,\, \RR^3\setminus\overline{B}, \\
\nu\times(\bE-\bF)=0\quad& \mbox{on}\,\, \p B,\\
\int_{\p B}\nu\cdot\epsilon (\bE-\bF)\,ds=0,\quad &\\
(\bE-\bF)(\Bx)=\Ocal(\|\Bx\|^{-2})\quad&\mbox{as}\,\, \|\Bx\|\rightarrow \infty.
\end{split}
\right.
\eeq
By using Lemma \ref{le:revisele02} there holds
\beq\label{eq:le01tmp02}
\bE-\bF=\nabla u \quad \mbox{in}\,\, \RR^3\setminus\overline{B}
\eeq
for some $u\in H^1_{loc}(\RR^3\setminus\overline{B})$ and $u(\Bx)=\Ocal(\|\Bx\|^{-1})$ as $\|\Bx\|\rightarrow \infty$.
By taking divergence of both sides of \eqnref{eq:le01revise01}, there further holds
\beq\label{eq:le01tmp03}
\nabla\cdot((1+\tilde\epsilon)\bE-\bF)=0\quad \mbox{in}\,\, \RR^3\setminus\overline{B}.
\eeq
Combining \eqnref{eq:le01tmp02} and \eqnref{eq:le01tmp03}, one can find that \eqnref{eq:le01tmp01} is equivalent to
\beq\label{eq:le01tmp04}
\left\{
\begin{split}
\nabla\cdot(\epsilon\nabla u)=-\epsilon_0\nabla\cdot(\tilde\epsilon\bF)\quad & \mbox{in}\,\,\RR^3\setminus\overline{B}, \\
\nu\times\nabla u=0\quad &\mbox{on} \,\, \p B, \\
\int_{\p B}\nu\cdot\epsilon \nabla u\, ds=0,\quad& \\
u(\Bx)=\Ocal(\|\Bx\|^{-1})\quad &\mbox{as}\,\, \|\Bx\|\rightarrow \infty.
\end{split}
\right.
\eeq
Note that $\nu\times\nabla u=0$ on $\p B$ is equivalent to $u=C$ on $\p B$, where $C$ is a constant.
Thus \eqnref{eq:le01tmp04} is uniquely solvable and so does \eqnref{eq:le01tmp01}.

The proof is complete.
\end{proof}

The following asymptotic expansion results hold immediately by using \eqnref{eq:revise001}, and Lemmas \ref{le:revisele01}, \ref{le:00} and \ref{le:01}.
\begin{lem}
Let $(\bE, \bH)$ be the solution to system \eqnref{eq:pss} and \eqnref{eq:radia2} with $\epsilon$ satisfy \eqnref{eq:le01revise001}.
Then there holds the following in $\Omega\setminus\overline{B}$
  \beq\label{eq:solforPEC01}
  \bE=(I_3-\mathfrak{K}_1(\epsilon))^{-1}\mathfrak{C}_1[\mathbf{p}]
  +i\omega(I_3-\mathfrak{K}_1(\epsilon))^{-1}\mathfrak{C}_1[(\Bx\cdot\mathbf{d})\mathbf{p}]+\Ocal(\omega^2),
  \eeq
  and
  \begin{equation*}
  \bH=i\mathfrak{C}_2[(\Bx\cdot\mathbf{d})\mathbf{p}]+\Ocal(\omega)\quad \mbox{as}\,\,\epsilon\rightarrow 0.
  \end{equation*}
\end{lem}

\section{Unique recovery results}

In this section, we present the main uniqueness results in recovering both the embedded obstacle and its surrounding medium.

In what follows, we let $(\Omega_j\setminus \overline{B_j}, \epsilon_j)$, $j=1,2$, be two EM configurations. For any fixed incident direction $\mathbf{d}$ and polarization $\mathbf{q}$, we let $\bE^{\infty}_j(\hat{\Bx},\omega)$ respectively signify the electric far-field patterns for the system \eqnref{eq:pss}-\eqnref{eq:radia2} associated with $(\Omega_j\setminus \overline{B_j}, \epsilon_j)$, $j=1,2$.


Let $B_{12}^{c}$ be the unbounded connected component of $\RR^3\setminus \overline{B_1\cup B_2}$.
Similarly, let $\Omega_{12}^{c}$ be the unbounded connected component of $\RR^3\setminus \overline{\Omega_1\cup \Omega_2}$.
If $B_1\neq B_2$, we know that either $(\RR^3\setminus \overline{B_{12}^{c}})\setminus\overline{B_1}$ or $(\RR^3\setminus \overline{B_{12}^{c}})\setminus\overline{B_2}$ is nonempty. In such a case, the domains $B_1$ and $B_2$ are said to be {\it admissible} if there exists a connected component, say $B^*$, of $(\RR^3\setminus \overline{B_{12}^{c}})\setminus\overline{B_1}$ or $(\RR^3\setminus \overline{B_{12}^{c}})\setminus\overline{B_2}$
such that the divergence theorem holds in $B^*$. Here, we note that divergence theorem always holds in Lipschitz domains (cf. \cite{Mclean}).
It is easily seen that $\partial B^*$ is composed of finitely many Lipschitz pieces.
Hence, one can see that if $B^*$ can be decomposed into the union of finitely many Lipschitz subdomains,
then the divergence theorem holds in $B^*$ and therefore $B_1$ and $B_2$ are admissible. It is also interesting to note a particular case that if both $B_1$ and $B_2$ are polyhedral domains, then $B^*$ is also a polyhedral domain. Hence, in the polyhedral case, both $B_1$ and $B_2$ are clearly admissible. Throughout the rest of the paper, we assume that, if $B_1\neq B_2$, then $B_1$ and $B_2$ are two admissible PEC obstacles.

\subsection{Recovery of the obstacle}
\begin{thm}\label{thm:mo1}
Assume that the electric permittivities $\epsilon_j$ are constants on $\p B_j, \,j=1,2$. Let $\omega_0$ be a positive number.
For any fixed incident direction $\mathbf{d}$ and polarization $\mathbf{q}$,
 if
\beq\label{eq:th031}
\bE^{\infty}_1(\hat{\Bx},\omega) = \bE^{\infty}_2(\hat{\Bx},\omega)
\eeq
for all observation directions $\hat{\Bx}\in\mathbb{S}^2$ and all frequencies  $\omega\in (0, \omega_0)$,
then we have $B_1=B_2$.
\end{thm}
\begin{proof}
For simplicity, we set $\bH_{j,0}:=i\mathfrak{C}_2^{(j)}[(\Bx\cdot\mathbf{d})\mathbf{p}]$ in $\RR^3\setminus\overline{B_j}$, $j=1, 2$.
Since by Rellich's lemma \cite{CK}, the far-field pattern uniquely determine the scattered field,
we deduce from the assumption \eqnref{eq:th031} that $\bH_1=\bH_2$ in $\Omega_{12}^{c}$ and thus
\beq\label{eq:HjequalOmege12}
\bH_{1,0}=\bH_{2,0}\quad \mbox{on}\,\, \Omega_{12}^{c}.
\eeq
Recall the definition of $\mathfrak{C}_2^{(j)}$ in \eqnref{eq:defopcon01} one has
\beq\label{eq:le32revise01}
\mathfrak{C}_2^{(j)}[(\Bx\cdot\mathbf{d})\mathbf{p}]=i\mu_0^{-1}D^2\Acal_{B_j}^{0}\Big(\frac{I}{2}+\Mcal_{B_j}^{0}\Big)^{-1}[(\Bx\cdot\mathbf{d})\nu\times\mathbf{p}]
-i\mu_0^{-1}\mathbf{d}\times\mathbf{p}.
\eeq
Straightforward calculations show that $\Delta \bH_{j,0}= 0$ in $\RR^3\setminus\overline{B_j}$, which in turn implies that $\bH_{j,0}$ is analytic in $\RR^3\setminus\overline{B_j},\,j=1,2$.
From \eqref{eq:HjequalOmege12}, by analytic continuation, we have
\beq\label{eq:revisethob001}
\bH_{1,0}=\bH_{2,0}\quad \mbox{on}\,\, B^c_{12}.
\eeq
This further implies $\nu\times \bH_{1,0}=\nu\times\bH_{2,0}$ on $\p B^c_{12}$.

If $B_1\neq B_2$, without loss of generality we assume that $B^*:=(\RR^3\setminus\overline{B_{12}^c})\setminus{B_1}$ and $B^*$ is nonempty.
By \eqnref{eq:le32revise01} one has
\begin{equation*}
\bH_{j,0}=\nabla \tilde u_j \quad\mbox{in} \,\, B_{12}^{c},
\end{equation*}
where
\beq\label{eq:reconobrevise0041}
\begin{split}
\tilde u_j:=&\mu_0^{-1}(\mathbf{d}\times\mathbf{p})\cdot\Bx -\mu_0^{-1}\nabla\cdot\Acal_{B_j}^{0}\Big(\frac{I}{2}+\Mcal_{B_j}^{0}\Big)^{-1}[(\Bx\cdot\mathbf{d})\nu\times\mathbf{p}]\\
=&\mu_0^{-1}(\mathbf{d}\times\mathbf{p})\cdot\Bx -\mu_0^{-1}\Scal_{B_j}^{0}\Big(-\frac{I}{2}+(\Kcal_{B_j}^{0})^*\Big)^{-1}[(\mathbf{d}\times\mathbf{p})\cdot\nu ], \quad j=1,2.
\end{split}
\eeq
From this, using jump relations, we find that
\beq\label{eq:reconobrevise0043}
\nu\cdot\nabla\tilde u_j|_+=0 \quad \mbox{on}\,\, \p B_j,\quad j=1,2.
\eeq
By \eqref{eq:revisethob001}-\eqref{eq:reconobrevise0043} we see that $\tilde{u}_1$ is a harmonic function in $B^*$ and satisfies the homogeneous Neumann boundary condition
$\nu\cdot\nabla\tilde u_1=0$ on $\p B^*$.
 By the Gauss divergence theorem, we have $\tilde u_1=C$ for some constant $C\in\mathbb{C}$ in $B^*$. Using the unique continuation of the harmonic function $\tilde u_1$, one has $\tilde u_1=C$ in $B_{12}^c$, which is a contradiction to \eqnref{eq:reconobrevise0041}. Thus $B_1=B_2$.

 The proof is complete.
\end{proof}

We would like to remark that the unique determination of the perfect conductor $B$ has been proved as long as the electric permittivity $\epsilon$ is a constant on $\p B$. That is, $\epsilon$ could be a variable function inside $\Omega\backslash\overline{B}$.

\subsection{Recovery of the surrounding medium}

In this subsection, we show that if $\epsilon$ is a constant in $\Omega\setminus\overline{B}$, then it can also be uniquely determined by using the same far-field data as those in the previous subsection. First, by Theorem~\ref{thm:mo1}, we readily have that embedded obstacle is uniquely recovered, namely $B=B_1=B_2$. Assume further that $\Om=\Om_1=\Om_2$. Before proceeding with further analysis, we recall the following inner transmission condition result (see \cite{DHU:17}),
\begin{lem}
Let $(\bE, \bH)$ be the solution to \eqnref{eq:pss} and \eqnref{eq:radia2}. Then there holds the following transmission condition, i.e.,
\beq\label{eq:addinnertran1}
\nu\cdot \epsilon\bE |_+= \nu\cdot \epsilon\bE |_- \quad \mbox{on}\,\, \p \Omega,
\eeq
\end{lem}

\begin{thm}\label{thm:mo2}
Suppose that $\epsilon_j$ are constant in $\Omega\setminus \overline{B}$. For any fixed incident direction $\mathbf{d}$ and polarization $\mathbf{q}$,
 if
\beq\label{eq:th032}
\bE^{\infty}_1(\hat{\Bx},\omega) = \bE^{\infty}_2(\hat{\Bx},\omega)
\eeq
for all observation directions $\hat{\Bx}\in\mathbb{S}^2$ and all frequencies  $\omega\in (0, \omega_0)$,
then $\epsilon_1=\epsilon_2$.
\end{thm}
\begin{proof}
Using Rellich's lemma \cite{CK}, we deduce from the assumption \eqref{eq:th032} that
$$
\bE_1=\bE_2\quad\mbox{in}\,\,\RR^3\setminus \overline{\Omega}.
$$
Note that $\epsilon = \epsilon_0$ in $\RR^3\setminus \overline{\Omega}$, by using the transmission condition \eqnref{eq:addinnertran1} one has
\beq\label{eq:revisetmp0000011}
\nu\cdot \bE_1|_-=\epsilon_1^{-1}\epsilon_2\nu\cdot \bE_2|_- \quad \mbox{on}\,\, \p \Omega.
\eeq
By using \eqnref{eq:lerevise0201}, \eqnref{eq:revisetmp00011} and the PEC condition on $\p B$, one can set $\bE_j=\nabla u_j +\Ocal(\omega)$, where $u_j$ are harmonic functions in $\Omega\setminus\overline{B}$ and are independent of $\omega$ and $u_1=u_2=0$ on $\p B$.
Since $\epsilon_j$, $j=1, 2$ are constants in $\Omega\setminus\overline{B}$, using the boundary condition \eqref{eq:revisetmp0000011}, we conclude that the difference $u:=u_1-\epsilon_1^{-1}\epsilon_2 u_2$ solves the following system
\beq
\left\{
\begin{split}
&\Delta u= 0 \quad \mbox{in}\,\, \Omega\setminus\overline{B}, \\
&u = 0 \quad \mbox{on}\,\, \p\, B,\\
&\nu\cdot\nabla u = 0 \quad \mbox{on}\,\, \p\,\Omega.
\end{split}
\right.
\eeq
Using Gauss divergence theorem, one immediately has $u=0$ in $\Omega\setminus\overline{B}$ and thus
\beq\label{eq:revisetmp0000012}
u_1=\epsilon_1^{-1}\epsilon_2 u_2 \quad \mbox{in}\,\, \Omega\setminus\overline{B}.
\eeq
By using \eqnref{eq:solforPEC01} one further has
\begin{equation*}
\nabla (u_1-u_2)=\tilde\epsilon_1\mathfrak{K}_1[\nabla u_1]-\tilde\epsilon_2\mathfrak{K}_1[\nabla u_2]\quad \mbox{in}\,\, \Omega\setminus\overline{B},
\end{equation*}
which together with \eqnref{eq:revisetmp0000012} yields
\beq\label{eq:revisetmp00122}
(\epsilon_2-\epsilon_1)\nabla u_2=(\tilde\epsilon_1\epsilon_2-\epsilon_1\tilde\epsilon_2)\mathfrak{K}_1[\nabla u_2]\quad \mbox{in}\,\, \Omega\setminus\overline{B}.
\eeq
If $\epsilon_1\neq\epsilon_2$, then by using \eqnref{eq:revisetmp00122} one has
\begin{equation*}
\nabla u_2=-\mathfrak{K}_1 [\nabla u_2]\quad \mbox{in}\,\, \Omega\setminus\overline{B},
\end{equation*}
which together with \eqnref{eq:solforPEC01} and \eqnref{eq:revisetmp00011} gives
\beq\label{eq:revisetmp00124}
\epsilon_0^{-1}\epsilon_2\nabla u_2=\mathfrak{C}_1[\mathbf{p}]\quad \mbox{in}\,\, \Omega\setminus\overline{B}.
\eeq
By the definition of $\mathfrak{C}_1[\mathbf{p}]$ and Lemma 5.5 in \cite{ADM}, one then has
\beq\label{eq:revisetmp00011}
\mathfrak{C}_1[\mathbf{p}]=-\nabla \Scal_B^0\left(\frac{I}{2}+(\Kcal_B^0)^*\right)^{-1}[\nu\cdot \mathbf{p}]+\mathbf{p}.
\eeq
Noting that $\mathfrak{K}_1 [\epsilon_2\nabla u_2]$ is a harmonic function in $\RR^3\setminus\overline{B}$ and decays as $\|\Bx\|^{-2}$ at infinity, and by \eqnref{eq:revisetmp00124} and unique continuation,  one has
\begin{equation*}
-\nabla \Scal_B^0\left(\frac{I}{2}+(\Kcal_B^0)^*\right)^{-1}[\nu\cdot \mathbf{p}]+\mathbf{p}=\Ocal(\|\Bx\|^{-2}) \quad \mbox{as}\,\, \Bx\rightarrow \infty,
\end{equation*}
which is a contradiction. Thus $\epsilon_1=\epsilon_2$.

The proof is complete.
\end{proof}

Clearly, by Theorems~\ref{thm:mo1} and \ref{thm:mo2}, an EM scatterer of the form $(\Omega\backslash\overline{B}, \epsilon)$ with both the embedded PEC obstacle $B$ and the constant permittivity $\epsilon$ being unknown, can be uniquely recovered by the multiple-frequency far-field data as specified in \eqref{eq:th031} or \eqref{eq:th032}.
Finally, we would like to remark on the measurements. Throughout, we have made use of the far-field data for the recovery of the inverse scattering problem. However, all the results equally hold when the far-field data are replaced by the near-field data, consisting of the tangential components of the electric and magnetic fields measured on any open surface away from the scatterer. Indeed, by the unique continuation principle for the Maxwell system \cite{CK}, those two sets of data are equivalent to each other. 

\section*{Acknowledgment}
The work of Y. Deng was supported by NSF grant of China, No. 11601528, NSF grant of Hunan No.2017JJ3432, Mathematics and Interdisciplinary Sciences Project, Central South University. The work of H. Liu was supported by the startup fund and FRG grants from Hong Kong Baptist University, and Hong Kong RGC General Research Funds, 12302415 and 12302017.
The work of X. Liu was supported by the Youth Innovation Promotion Association CAS and the NNSF of China under grant 11571355.

\end{document}